\documentclass[12pt,reqno]{amsart}

\usepackage{amsfonts,amsmath,amsthm}
\usepackage{amssymb,epsfig}
\usepackage{enumerate}

\usepackage[text={425pt,650pt},centering]{geometry}
\usepackage{graphicx}
\usepackage{epsfig}
\usepackage{tikz}
\usepackage{caption}
\usepackage{color} 
\definecolor{vert}{rgb}{0,0.6,0}

\numberwithin{figure}{section}

\theoremstyle{plain}
\newtheorem{thm}{Theorem}[section]
\newtheorem{ass}{Main Assumptions}

\newtheorem{defn}{Definition}

\newtheorem{lem}[thm]{Lemma}

\newtheorem{prop}[thm]{Proposition}
\theoremstyle{remark}

\numberwithin{equation}{section}

\newtheorem*{claim*}{Claim}

\newtheorem*{proofofclaim*}{Proof of claim}

\newcommand{\R}{\mathbb{R}}

\newcommand{\W}{W^{1,\infty}}

\newcommand{\gam}{\gamma}

\newcommand{\ep}{\varepsilon}

\newcommand{\ol}{\overline}

\begin{document}

\title
{On the Uniqueness for One-Dimensional Constrained Hamilton-Jacobi Equations}

\author
{Yeoneung Kim}

\begin{abstract}

The goal of this paper is to study uniqueness of a one-dimensional Hamilton-Jacobi equation
\begin{equation*}
	\begin{cases}
	u_t=|u_x|^2+R(x,I(t)) &\text{in }\R \times (0,\infty), \\
	\max_{\R} u(\cdot,t)=0 &\text{on }[0,\infty),
	\end{cases}
\end{equation*}
with an initial condition $u_0(x,0)=u_0(x)$ on $\R$. A reaction term $R(x,I(t))$ is given while $I(t)$ is an unknown constraint (Lagrange multiplier) that forces maximum of $u$ to be always zero. In the paper, we prove uniqueness of a pair of unknowns (u,I) using dynamic programming principle in one dimensional space for some particular class of nonseparable reaction $R(x,I(t))$.

\end{abstract}

\thanks{Supported in part by NSF grant DMS-1664424}

\address
{
Department of Mathematics,
University of Wisconsin Madison, 480 Lincoln Drive, Madison, WI 53706, USA}
\email{yeonkim@math.wisc.edu}

\date{\today}

\maketitle

\section{Introduction}
The non-local parabolic equations arising in adaptive dynamics(see \cite{Darwin1, Darwin2, Darwin3, Darwin4}) have an interesting feature so called Dirac concentration of density as a diffusion coefficient vanishes. To illustrate this, we consider the following evolution equation
\begin{equation}
	\begin{cases}
	n^\ep_t - \ep \Delta n^\ep= \frac{n^\ep}{\ep} R(x,I^\ep(t)) &\text{in } \R^n \times (0,\infty),\\
	n^{\ep}(x,0)=n^\ep_0 \in L^1(\R^n) &\text {on }\R^n,\\
	I^\ep(t)=\int_{\R^n} \psi(x) n^\ep(t,x)dx,
	\end{cases}
\end{equation}
where the spatial variable $x$ denotes `traits' in the environment. Furthermore, $n^\ep$, $R(x,I^\ep(t))$, $\ep$ and $\psi(x)$ describe density of the population, reproduction rate, mutation rate and consumption rate by a trait $x$. Here $\psi$ assumed to be a nonnegative compactly supported function. We then take Hopf-Cole transformation $n^\ep(x,t)=e^{u^\ep(x,t) /\ep}$. It was shown in many literatures that as mutation rate $\ep$ vanishes, $u^\ep$ converges locally uniformly to $u$ which is a viscosity solution to
\begin{equation}\label{origin}
	\begin{cases}
	u_t=|Du|^2+R(x,I(t)) &\text{in } \R^n \times (0,\infty),\\
	\max_{ \R^n} u(\cdot,t)=0 &\text{on } [0,\infty),\\
	u(x,0)=u_0(x) &\text{on } \R^n.
	\end{cases}
\end{equation}
The constraint of $u$ is obtained from the property that $I^\ep$ is positive and uniformly bounded. It was also shown that
\begin{equation}
n^\ep(x,t) n(x,t) \rightharpoonup {\ol \rho}(x)(x(t)-\ol x(t))\text{ weakly in the sense of measure}
\end{equation}
where
\begin{equation*}
u(\ol x(t),t)=\max_{\R} u(\cdot,t)=0 \text{ and } \rho(t)=\frac{I(t)}{\psi(x)}
\end{equation*}
for the solution $n^\ep(x,t)$ to (\ref{origin}) (see \cite{Dirac.C, Convergence}). Despite the existence of solutions to (\ref{origin}) is quite well understood, the uniqueness is relatively less known. In the recent work by S. Mirrahimi, J. -M. Roquejoffre \cite{Uniqueness}, the uniqueness of the solution is shown when the reaction and initial condition $u_0(x)$ are strictly concave so that regularity of maximum point is obtained. However, the uniqueness for general initial data and a nonconave reaction is still open. In this paper, the uniqueness property for constrained Hamilton-Jacobi equations in 1-D with some nonseparable reaction terms is obtained using dynamic programming principle.

\subsection{Setting and main result}
We need following assumptions on
\begin{equation}
R(x,I):\R \times [0,\infty) \rightarrow \R \text{ and } u_0(x): \R \rightarrow \R.
\end{equation}
where the reaction term is defined as
\begin{equation}
	R(x,I)=	
	\begin{cases}
	b(x)-Q(I) & \text{ for } x \geq 0,\\
	R'(x,I) & \text{ for } x <0.
	\end{cases}
\end{equation}

\begin{ass}
	\begin{itemize} For a positive $I_M$,
	\item [(A1)] $R$ is smooth and $R'(\cdot, I)<0$ on $(-\infty,0)$ for any positive $I$ ;
	\item [(A2)] $\sup_{0\leq I \leq I_M} \|R(\cdot,I)\|_{W^{2,\infty}} <\infty$ and $R$ is strictly decreasing in $I$;
	\item [(A3)] $Q(I)\geq 0$ is strictly increasing in $I$ and $Q(0)=0$
	\item [(A4)] $\sup_{\R} R(\cdot,I_M)=0$;
	\item [(A5)] $\min_{\R} R(\cdot,0)=0$;
\item [(A6)] $b(x)$ is strictly increasing on $[0,\infty)$ with b(0)=0;
\item [(A7)] $b'(x)$ is Lipschitz continuous, hence, nonnegative ;
\item [(A8)] $u_0(x) \in C^2(\R)$ with $\|u_0\|_{C^2(\R)}<\infty$, $\max_{x\in \R} u_0(\cdot)=u_0(0)=0$ and $u_0(x)<0$ elsewhere.
	\end{itemize}

\end{ass}

Additionally, $f \in \W(\R^n)$, that is; $\|f\|_{L^\infty(\R^n)}+\|Df\|_{L^\infty(\R^n)} <\infty$.

Now we are ready to state our main theorem. Under the assumptions above, we consider the following equation.
\begin{equation}\label{HJ}
	\begin{cases}
	u_t={u_x}^2+R(x,I(t)) &\text{in } \R \times (0,\infty),\\
	\max_{ \R} u(\cdot,t)=0 &\text{on } [0,\infty),\\
	u(x,0)=u_0(x) &\text{on } \R.
	\end{cases}
\end{equation}

\begin{thm} \label{main}
There exists at most one pair $(u,I)$ such that $u(x,t) \in C(\R \times (0,\infty))$ solves (\ref{HJ}) in viscosity sense and $I(t) \in C([0,\infty))$ is strictly increasing.
\end{thm}

\section{Preliminary}
Throughout the section, let us assume $(u,I)\in C(\R \times (0,\infty)) \times C([0,\infty))$ is a pair of solution to (\ref{HJ}) in viscosity sense. By a Lipschitz estimate provided by the author in \cite{yeon}, one can assume further that $u$ is Lipschitz continuous in $\R \times [0,T]$ for any positive $T$. Now we follow dynamic programming principle arguments presented in \cite{Uniqueness}, which yields
\begin{equation}
u(x,t)=\sup_{\gam(t)=x} \{F(\gam):\gam \in AC([0,t];\R) \}
\end{equation}
where
\begin{equation}
F(\gam):=u_0(\gam(0))+\int_0 ^t \left( -\frac{ \dot \gam ^2}{4} +R(\gam(s),I(s) \right) ds.
\end{equation}
Furthermore, one can actually show that there exists a path $\gam(s) \in C^1([0,t);\R))$ such that
\begin{equation}
u(x,t)=u_0(\gam(0))+\int_0 ^t \left( -\frac{ \dot \gam ^2}{4} +R(\gam(s),I(s) \right) ds
\end{equation}
with $\gam(t)=0$ and it satisfies Euler-Lagrange equation
\begin{equation}\label{EL}
\begin{cases}
\ddot \gam(s)+2R_x(\gam(s),I(s))=0,\\
\dot \gam(0)+2\dot {u_0}(\gam(0))=0,\\
\gam(t)=x.
\end{cases}
\end{equation}
For the details, see \cite {Uniqueness} and references therein.

There could be more than one solution to the equation above. However, the Euler-Lagrange equation reduces to a simpler equation that results in the existence of a unique solution in our setting. We start with some generic properties.
\begin{prop}\label{prop1}
Assume that $\max_{\R} u(\cdot, t)=u(x',t)=0$. Then $R(x',I(t))$=0.
\end{prop}
\begin{proof}
By viscosity subsolution test, one can easily obtain $R(x',I(t) \geq 0$. Now we assume that $R(x',I(t))>0$. Then there exists $t_0>0$ such that $R(x',I(s))>0$ on $[t,t+t_0]$ by the continuity of $I$ and $R$. Integrating (\ref{HJ}) both sides over $\{x'\} \times [t,t+t_0]$ yields
\begin{align*}
u(x',t+t_0)-u(x',t) & \geq \int_t ^{t+t_0} R(x',I(s))ds >0
\end{align*}
Hence, we get
\begin{equation*}
u(x',t+t_0) >0,
\end{equation*}
which violates the maximum constraint.

\end{proof}
\begin{defn} We define $x(t) \in \R$ to satisfy
\begin{equation*}
R(x(t),I(t))=0
\end{equation*}
for $t>0$ and a strictly increasing $I(t)$. Then, together with Propositition \ref{prop1}, we have
\begin{equation}\label{p}
max_{\R} u(\cdot,t)=u(x(t),t))=0.
\end{equation}
for a solution pair $(u,I)$.
\end{defn}

\begin{prop}\label{bd}
$I(0)=0$ and $I(s) \leq I_M$ on $[0,\infty)$.
\end{prop}
\begin{proof}
Let us first prove $I(0)=0$ when $(u,I)$ is a pair of solution. We may assume $I(0)>0$. From the property (\ref{p}), we deduce
\begin{equation*}
0=\lim_{t\rightarrow0+} u(x(t),t)=u(x(0+),0)<0
\end{equation*}
where $x(0+)$ is a right limit of $x(t)$, which yields contradiction. Therefore, $I(0)=0$. The second part of the proposition, $I(s) \leq I_M$, is a straight consequence of Proposition \ref{prop1} due to the assumption on $R$.
\end{proof}

We also need some regularity properties of the solution $u(x,t)$, which play crucial roles in analyzing the trajectory $\gam(s)$.

\begin{defn} For a real valued function $u(x)$ define for $x\in \R^n$, we define super differential and sub differential at $x$ as
\begin{align}
D^{+}u(x)&=\{p \in \R^n : \liminf_{y \rightarrow x} \frac{u(y)-u(x)- p \cdot (y-x)} {|y-x|} \geq 0\}	\\
D^{-}u(x)&=\{p \in \R^n : \limsup_{y \rightarrow x} \frac{u(y)-u(x)-p\cdot (y-x)}{|y-x|} \leq 0\}
\end{align}
\end{defn}

\begin{lem}\label{reg}
A solution $u(x,t)$ is semiconvex in $x\in \R $ for any fixed positive $T$.
\end{lem}
\begin{proof}
Let us define $v(x,t)=-u(x,t)$ and prove $v(x,t)$ is semiconcave in $\R \times [0,T]$. Cleary, $v$ satisfies
\begin{equation}
	\begin{cases}
	v_t+v_x^2 +R(x,I(t))=0 &\text{ in } \R^n \times (0,T],\\
	v(x,0)=-u(x,0) &\text{ on }\R
	\end{cases}
\end{equation}
in viscosity sense. To prove semiconcavity of $v$, we first provide a priori estimate for $v^\ep$ where $v^\ep$ is a unique solution to
\begin{align}\label{ap}
	\begin{cases}
	v_{t}^\ep + (v_{x}^{\ep})^2 +R(x,I(t))=\ep v_{xx}^\ep &\text{ in } \R \times (0,T],\\
	v^\ep(x,0)=-u_0(x):=v_0(x) &\text{ on } \R.
	\end{cases}
	\end{align}
Differentiating (\ref{ap}) twice with respect to $x$ and substituting $w$ for $v_{xx}^\ep$ yields
\begin{equation}
w_t+2w^2+2v_x w_x+R_{xx}=\ep w_{xx}.
\end{equation}
It is known that $w$ is bounded but the bound depends on $\ep$. However, one can actually show that the bound is uniform in $\ep$. To justify this, we first notice that $w$ is a subsolution to the following parabolic equation
\begin{align}\label{para}
\begin{cases}
w_t+v_x w_x +R_{xx}=\ep w_{xx} &\text{ in } \R \times (0,T],\\
w(x,0)=v''_0(x) &\text{ on } \R.
\end{cases}
\end{align}
On the other hand, $v_0+Ct$ and $v_0-Ct$ are supersolution and subsolution to (\ref{para}) respectively where $C$ depends only on the bound for $R_{xx}$. Therefore, by comparison principle, one can obtaiin $|w|<C$ where $C$ does not depend on $\ep$.
As a last step, we need the following estimate.
\begin{claim*}
\textit{There exists positive $C$ that depends only on $T$ such that}
\begin{equation*}
	\|v_{t}^\ep \|_{L^\infty(\R \times [0,T])} +\|v_{x}^\ep\| _{L^\infty(\R \times [0,T])} < C
	\end{equation*}
\end{claim*}
\begin{proofofclaim*}
Since $0\leq I(t) \leq I_M$, $R(x,I)$ is bounded, for $C>0$ large enough, we have
\begin{equation}
v_0(x)-Ct \leq v^\ep(x,t) \leq v_0(x)+Ct
\end{equation}
by the comparison principle. Using the comparison principle one more time yields
\begin{equation}
v^\ep(s+t) \geq v^\ep(t)-Cs
\end{equation}
for $s,t \geq 0$. Therefore, $v_{t}^\ep > -C$ in $\R \times [0,T]$. On the other hand, observing the original equation (\ref{ap}), we can derive $\|v_{x} ^\ep \|_{L^\infty(\R \times [0,T])}< C$ as $v_{xx} ^\ep$ is bounded above. Finally, an upper bound for $v_{t}^\ep$ is obtained, and such bounds depend only on $T$.
\end{proofofclaim*}
As a consequence, $v_\ep$ converges locally uniformly to $v$ as $\ep$ goes to $0$ by Arzela-Ascoli and by the uniqueness and stability of a viscosity solution. Moreover, the semiconcaivity of $v^\ep$ in $x$ implies that
\begin{equation*}
v^\ep(x,t)-K|x|^2
\end{equation*}
is concave in $x$ for some positive $K$. Combining it with locally uniform convergence of $v^\ep$, we get semiconcavity of $v$ in $x$. Thereroe, $u$ is locally semiconvex in $x$.
\end{proof}
\begin{lem}\label{lemma}
For each $t \in (0,\infty)$, $u(x,t)$ is differentiable at $(x(t),t)$ with respect to the space variable $x$ and it satisfies
\begin{equation}\label{re}
0=u_x(x(t),t)=-\frac{\dot \gam_x(t)}{2}.
\end{equation}
In addition to that, by the maximum constraint, we have
$\dot \gam_x(t)=0$.
\end{lem}

\begin{proof}
By Lemma \ref{reg}, $v(x,t)=-u(x,t)$ is semiconvcave in $x$. Hence, supper differential at $(x(t),t)$ is nonempty. On the other hand $p=0$ is a subdifferential of for $v$ at $(x(t),t)$. Therefore, $u$ is differentiable with respect to the space variable at $(x(t),t)$. Moreover, the derivative is $0$.

A classical result in \cite{Concave} suggests that
\begin{equation*}
\eta(t) \in {\nabla}^+ v(x(t),t)
\end{equation*}
where $\dot \gam_x(s)= 2\eta(s)$ for $s\in[0,t]$ and $v$ is defined as above. Combining these two, we get the result using the differentiability of $v$ at $(x(t),t)$.
\end{proof}

\begin{prop}\label{imp}
Let $\gam(s)\in C^1([0,t];\R)$ an optimizing path whose terminal point is $x(t)$ and $x(s)\in \R$ satisfy $R(x(s),I(s))=0$ for $s>0$. Then we have $\gam(s)>x(s)$ for $s\in(0,t)$.
\end{prop}

\begin{proof}
We may assume first that $\gam(s) \geq 0$ since $F(\gam^+) \geq F(\gam)$ where
\begin{align}\gam^+(s)=
\begin{cases}
\gam(s) &\text{if }\gam(s)>0\\
0 &\text{if } \gam(s) \leq 0
\end{cases}
\end{align}
Now we assume $\gam(s)<x(s)$ on $(0,t)$. Then $R(\gam(s),I(s))<R(x(s),I(s))=0$ on $(0,t)$, which yields
\begin{equation*}
0=u(x(t),t)= \int_{t_0} ^{t} \left( -\frac{ \dot \gam ^2}{4} +R(\gam(s),I(s) \right) +u_0(\gam(0))<0.
\end{equation*}
Hence, There exists $t' \in (t_0,t)$ such that $\gam(t')=x(t')$.
On the other hand, $\gam(s)$ satisfies the Euler-Lagrange equation, which is,
\begin{equation}\label{el}
\ddot \gam(s) + R_x(\gam(s),I(s))=\ddot \gam(s) + b'(\gam(s))=0.
\end{equation}
Integrating the equation from $t'$ to $t$ gives
\begin{equation*}
0=\dot \gam(t)-\dot \gam(t_0) = \int_{t_0}^{t} b'(\gam(s)) >0,
\end{equation*}
by the lemma above. Therefore, $\gam(s)>x(s)$ on $(0,t)$.
\end{proof}

\section{Proof of the theorem \ref{main}}
We assume that we have two pairs of solutions $(u_1,I_1)$ and $(u_2,I_2)$ to (\ref{HJ}) for $n=1$ and consider two cases. Let us fix the time $T$.

{\it{Case 1}} : $I_1(s)$ and $I_2(s)$ intersect only at the origin for $s\in [0,T]$.\\
Without loss of generality, let us assume $I_1<I_2$ except for the terminal point.
Then $u_1$ is a viscosity supersolution to
\begin{equation}\label{comp}
	\begin{cases}
	(u_2)_t={(u_2)_x}^2+R(x,I_2(t)) &\text{in } \R \times (0,t],\\
	u(x,0)=u_0(x) &\text{on } \R.
	\end{cases}
\end{equation}
By the comparison principle and the maximum constraint, we have $x_1(s)=x_2(s)$ for all $s$, where $x_1,x_2$ are defined as above, which is a contradiction.

{\it{Case 2}} : $I_1(s)$ and $I_2(s)$ intersect at more than one point including the terminal point $t$.
Let $t_0<t_1 \in [0,t]$ be points such that
\begin{equation}
I_1(t_i)=I_2(t_i) \text{ for } i=1,2.
\end{equation}
Hence, we have $x_1(t_0)=x_2(t_0):=\alpha$ and $x_1(t_1)=x_2(t_1):=\beta$. In addition to that, we may assume that
\begin{equation*}
I_1>I_2 \text{ for } i \in (t_0,t_1).
\end{equation*}
For the $t_i$'s above, we define $\gam_1(s)$ and $\eta_1(s)$ as optimizing trajectories corresponding to $I_1$ whose terminal points are $\alpha$ and $\beta$ respectively. Similarly, one can define $\gam_2(s)$ and $\eta_2(s)$ as optimizing trajectories corresponding to $I_2$ whose terminal points are $\alpha$ and $\beta$ respectively.
By Proposition \ref{imp} and Lemma \ref{lemma}, for each $i=1,2$, $\gam_i$ satisfies
\begin{equation*}
	\begin{cases}
	\ddot \gam_i +2b'(\gam_i) =0,\\
	\dot \gam_i(t)=0,\\
	\gam(t)=\alpha.
	\end{cases}
\end{equation*}
Similarly, for each $i=1,2$, $\eta_i$ is a solution to
\begin{equation*}\label{comp}
	\begin{cases}
	\ddot \eta_i +2b'(\eta_i) =0,\\
	\dot \eta_i(t)=0,\\
	\gam(t)=\beta.
	\end{cases}
\end{equation*}
Therefore, $\gam_1=\gam_2:=\gam$ and $\eta_1=\eta_2=\eta$. Applying this property to the relations
\begin{align*}
0=u_1(\beta,t_1)=\int_0^{t_1} \left(-\frac{\dot \gam^2 }{4} + b(\gam)-Q(I_1)\right)ds+u_0(\gam(0),0),\\
0=u_2(\beta,t_0)=\int_0^{t_1} \left(-\frac{\dot \gam^2 }{4} + b(\gam)-Q(I_2)\right)ds+u_0(\gam(0),0),\\
0=u_1(\alpha,t_1)=\int_0^{t_1} \left(-\frac{\dot \eta^2 }{4} + b(\gam)-Q(I_1)\right)ds+u_0(\eta(0),0),\\
0=u_2(\alpha,t_0)=\int_0^{t_1} \left(-\frac{\dot \eta^2 }{4} + b(\gam)-Q(I_2)\right)ds+u_0(\eta(0),0),
\end{align*}
we end up getting
\begin{equation*}
0=\int_{t_0}^{t_1} \left(Q(I_1)-Q(I_2)\right) ds,
\end{equation*}
which contradicts $I_1>I_2$ on $(t_0,t_1)$.

\end{document}